\documentclass[letterpaper, 10 pt, conference]{ieeeconf}  
\usepackage{mathtools} 
\usepackage{graphicx} 
\usepackage{amsmath} 
\usepackage{amssymb}  
\usepackage{multirow}
\usepackage{color}
\usepackage{algorithm, algorithmicx, algpseudocode}

\newcommand{\R}{{\mathbb R}}
\renewcommand{\L}{{\mathbb L}}

\newcommand{\Rnn}{{\mathbb R}_{\ge 0}}
\newcommand{\Rp}{{\mathbb R}_{> 0}}
\newcommand{\C}{{\mathbb C}}

\newcommand{\cC}{{\mathcal C}}
\newcommand{\cD}{{\mathcal D}}

\newcommand{\cK}{{\mathcal K}}

\newcommand{\cO}{{\mathcal O}}

\newcommand{\cQ}{{\mathcal Q}}

\newcommand{\cS}{{\mathcal S}}

\newcommand{\diag}[1]{{\mathrm{diag}}\{#1\}}

\newcommand{\bfone}{\mathbf 1}

\newcommand{\inter}{\mathrm{int}}

\newcommand{\Lone}{\mathrm L_{1}}
\newcommand{\Lp}{\mathrm L_{p}}

\newcommand{\Linf}{\mathrm L_{\infty}}

\renewenvironment{proof}[1]{{\it Proof:\,}}{\hfill\QED\par}
\def\QED{\mbox{\rule[0pt]{1.3ex}{1.3ex}}}
\newenvironment{proof-of}[1]{{\it Proof of #1:\,}}{\hfill\QED\par}

\newtheorem{thm}{Theorem}
\newtheorem{cor}{Corollary}
\newtheorem{defn}{Definition}

\newtheorem{prop}{Proposition}

\newtheorem{rem}{Remark}

\title{Properties of Eventually Positive Linear Input-Output Systems}
\author{Aivar Sootla}
\begin{document}
\maketitle
\begin{abstract}
\looseness=-1 In this paper, we consider systems with trajectories originating in the nonnegative orthant and becoming nonnegative after some finite time transient. These systems are called eventually positive and our results are based on recent theoretical developments in linear algebra. We consider dynamical systems (i.e., fully observable systems with no inputs), for which we compute forward-invariant cones and Lyapunov functions. We then extend the notion of eventually positive systems to the input-output system case. Our extension is performed in such a manner, that some valuable properties of classical internally positive input-output systems are preserved. For example, their induced norms can be computed using linear programming and the energy functions have nonnegative derivatives. We illustrate the theoretical results on numerical examples.
\end{abstract}
\maketitle

\section{Introduction}
Matrices with nonnegative entries (or nonnegative matrices) have received a considerable attention in literature~\cite{berman1994nonnegative} due to strong theoretical results such as Perron-Frobenius theorem describing their spectral properties. Nonnegative matrix theory has been later applied to dynamical systems in the context of positive systems (systems with nonnegative trajectories forward in time for nonnegative initial conditions). These systems also appear in practical applications such as economics~\cite{leontief1986input}, biology~\cite{sontag2007monotone} for clear reasons: the states have a physical interpretation prohibiting negative values. Positive systems with inputs have become a major topic in control theory since these systems naturally allow decentralized stability tests, decentralized control and model reduction algorithms~\cite{briat2013robust, rantzer2015ejc, tanaka2011bounded, Sootla2012positive}. 

In the input-output setting, \emph{externally positive systems} were studied from system-theoretic point of view~\cite{anderson1996nonnegative}. These systems are characterized by nonnegative outputs given nonnegative inputs and the zero initial condition, however, their state trajectories often are not positive for all future times and some state trajectories can become negative for some nonnegative initial conditions. Recently, it was shown in~\cite{altafini2015realizations} that there are externally positive systems with \emph{eventually} positive state trajectories, that is, trajectories from nonnegative initial conditions become (and stay) nonnegative after some initial transient. Eventual positivity stems from studying Perron-Frobenius theorem (see~\cite{elhashash2008general} and the references therein). In particular, eventually positive systems are fully characterized by this remarkable result, in particular, the strong version Perron-Frobenius theorem sets necessary and sufficient conditions for strong eventual positivity. Recently, eventual positivity was extended to arbitrary cones and matrix exponentials~\cite{olesky2009m,kasigwa2017eventual} preparing all the necessary ingredients for the control theoretic approach to eventual positivity, which is the main topic of this paper.

We first consider dynamical eventually positive systems, which are studied in~\cite{kasigwa2017eventual} from the linear algebra point of view and we offer a control-theoretic one. We show that any eventually positive system is also positive with respect to some cone, which however may be hard to characterize explicitly or it may be hard to use. We also derive straightforward formulas to compute invariant cones and Lyapunov functions for eventually positive systems under some additional assumptions on the system matrix. We then consider positive input-output systems, which are characterized by eventually positive trajectories of the states and positive outputs. We extend some of the properties of positive systems to this case such as: computation of induced norms, properties of energy functions and discuss implications for model reduction. 

It is worth mentioning that some of the strong properties of positive systems exploit scaled diagonal dominance of the system matrices rather than nonnegativity of the trajectories. In fact, it was shown that a subset of scaled diagonally dominant matrices admit diagonal Lyapunov functions~\cite{hershkowitz1985lyapunov}, which can be computed by linear programming/algebra~\cite{sootla2016existence}. These results were then extended to block-partitioned matrices in~\cite{sootla2017blocksdd}. Eventually positive systems can be considered as a complimentary extension of positive systems with a different set of retained properties.

The rest of the paper is organized as follows. In Section~\ref{s:even-pos}, we introduce notation and cover some preliminary results on positive and eventually positive systems. In Section~\ref{s:results} we present the theoretical results of the paper, which we illustrate on examples in Section~\ref{s:example}. We conclude in Section~\ref{s:con}.

\section{Preliminaries\label{s:even-pos}}
\subsection{Notation}
A diagonal matrix $A\in \C^{n\times n}$ with elements $a_i$ on its diagonal is denoted as $\diag{a_1, \dots, a_n}$. For a matrix $A \in \R^{n\times n}$, we assume that its eigenvalues $\{\lambda_1, \dots, \lambda_n \}$ (counted with their algebraic multiplicities) are ordered according to their real parts $\Re(\lambda_i) \ge \Re(\lambda_j)$ for all $i \le j$. The spectral radius $\rho(A)$ is defined as the maximum absolute value of its eigenvalues, that is $\rho(A) = \max\{|\lambda_1|, \dots, |\lambda_n| \}$. The spectral abscissa $\eta(A)$ is defined as the maximum real value of its eigenvalues, that is $\eta(A) = \max\{\Re(\lambda_1), \dots, \Re(\lambda_n) \}$. The matrix $A$ is called diagonalizable if the algebraic and geometric multiplicites of its eigenvalues are equal or equivalently, if there exists an invertible matrix $V$ such that $A = V \diag{\lambda_1, \dots, \lambda_n} W$, where $V^{-1} = W$, the columns $v^i$ (respectively, $w^i$) of the matrix $V$ (respectively, $W$) are the right (respectively, left) eigenvectors corresponding to the eigenvalues $\lambda_i$. We will write $A\ge 0$ for $A\in\R^{n\times n}$, if all the entries $A_{i j}$ of $A$ are nonnegative, $A>0$ if the matrix $A\ge 0$ and $A$ is not equal to zero, and $A\gg 0$, if all $A_{i j}$ are positive. The matrix $A^T$ denotes the transpose of $A$. We call the system $\dot x= A x$ invariant with respect to a set $\cS$, if the trajectories $x(t)\in\cS$ provided that the initial condition $x_0\in\cS$ or equivalently $e^{A t} \cS\subseteq \cS$ for all $t\ge 0$. 
A set $\cK$ is called a \emph{proper cone} if $\alpha v\in \cK$ for $v\in\cK$ and $\alpha\in\Rp$ (meaning that $\alpha \cK \subseteq \cK$ for all $\alpha \in\Rp$), $\cK + \cK \subseteq \cK$, $\cK\cap -\cK \subseteq \{ 0\}$, $\cK$ is closed and its interior (denoted as $\inter(\cK)$) is nonempty. Throughout this paper, we will employ only proper cones. We define the nonnegative orthant $\Rnn^n = \{x\in\R^n | x_i \ge 0, \, \forall i =1, \dots,n \}$ and the positive $\Rp^n$ as its interior. The cone $\cK = \{x\in\R^n| \sum_{i = 2}^n x_i^2 \le x_1^2, x_1 \ge 0 \}$ is called a Lorentz cone. The set $\cK^\ast$ denotes a dual set to $\cK$ namely $\cK^\ast = \{u\in \R^n | u^T v\ge 0,\, \forall v\in \cK \}$. We denote the boundary of a set $\cK$ as $\partial \cK$. We define the partial order using proper cones as follows, $x\preceq_\cK y$ if $x-y\in\cK$, we also write $x\prec_\cK y$ if $x\preceq_\cK y$ and $x\ne y$, and $x\ll_\cK y$ if $x-y\in\inter(\cK)$. If the order is induced by $\Rnn^n$, we drop the subscript and simply write $x\preceq y$, $x\prec y$, $x\ll y$.

\subsection{(Eventually) Positive Systems}
In this section, we present a characterization of the following class of linear dynamical systems. 
\begin{defn}\label{def:ev-pos}
The system $\dot x = A x$ is \emph{eventually $\cK$-positive} if  there exists $\tau_0\ge 0$ such that  $e^{A t} x \subseteq \cK$ for any $x \in \cK$ and for all $t\ge \tau_0$. The system is \emph{strongly eventually $\cK$-positive} if  $x\ne 0$ there exists $\tau_0\ge 0$ such that  $e^{A t} x\in\inter(\cK)$ for any $x \in \cK$ and for all $t\ge \tau_0$. If $\tau_0 = 0$, we call the systems simply $\cK$-positive and strongly $\cK$-positive, respectively. 
\end{defn}

Note that in our definition $\tau_0$ is chosen uniformly with respect to the vectors $x$.\footnote{In the previous version of our paper, we claimed that it is possible to switch the order of the predicates $\forall t\ge \tau_0$ and $\forall x$ if the cone $\cK$ is proper. This, however, is incorrect as explained in~\cite{gluck2023eventual}.} We can compactly define eventual $\cK$-positivity as $e^{A t} \cK \subseteq \cK$ for all $t\ge \tau_0$. If the cone $\cK$ is equal to $\Rnn^n$, then we refer to these systems as  (strongly) (eventually) positive. $\cK$-positive systems have been studied extensively in the literature and several necessary and sufficient conditions were proposed to certify positivity~\cite{hirsch2005monotone}.
\begin{prop}
	The system $\dot x = A x$ is $\cK$-positive if and only if $\lambda \in \cK^\ast$, $x \in \partial \cK$, $\lambda^T x = 0$ implies  that $\lambda^T A x \ge 0$. 
\end{prop}

Note that in the case of polyhedral cones $\cK$, the certificates reduces to a finite number of inequalities. Furthermore, in the case of $\Rnn^n$-positive systems (or more generally orthant-positive systems), additional powerful properties can be obtained. For example, in terms of computation of Lyapunov functions (cf.~\cite{rantzer2015ejc}). Positivity was also studied in the context of input-output systems such as
\begin{gather}\label{sys:in-out}
G =\left\{ \begin{aligned}
\dot x & = A x + B u\\
 y &= C x + D u
\end{aligned}\right.
\end{gather}
For example, the following class of systems received attention in the literature~\cite{anderson1996nonnegative}.
\begin{defn}
	A linear system is called \emph{externally positive}, if for the zero initial condition $x_0 = 0$, and any control signal $u(t)\ge 0$, we have that $y(t) \ge 0$ for all $t \ge 0$.
\end{defn}
It is well-known that the impulse response $C e^{A t} B + D$  is nonnegative for all $t>0$ if and only if the system is externally positive. Let $g(t) = (Ce^{A t} B + D)\delta(t)$, then the $p$-induced norm of $G$ is defined as follows:
\begin{gather*}
\|G \|_{\rm p-ind} = \sup\limits_{u \in\L_p^m[0,\infty)}\frac{\|g\ast u\|_p}{\|u\|_p},
\end{gather*}
where $\|u(t)\|_p = (\sum_k \int_0^\infty |u_k(t)|^p)^{1/p}$. The following remarkable result establishes straightforward (algebraic) computations of such norms:

\begin{prop}[Theorem~3 in~\cite{rantzer2015ejc}] \label{prop:induced-norms}
	Let $g(t) = (C e^{A t} B + D) \delta(t)$, where $C e^{A t} B \ge 0$ for $t\ge 0$ and $D\ge 0$, while $A$ is Hurwitz. Then $\|g\|_{\rm p-ind} = \|D - C A^{-1} B\|_{\rm p-ind}$ for $p =1$, $p =2$, $p=\infty$. In particular, if $g$ is scalar, then $\|g\|_{\rm p-ind} = |D - C A^{-1} B|$ for all $p\in[1,\infty]$.
\end{prop}

In this paper, we will focus on the following class of systems
\begin{defn}
	A linear system is called \emph{internally eventually positive}, if for any initial condition $x_0 \ge 0$, and any control signal $u(t)\ge 0$, we have that $y(t) \ge 0$ for all $t \ge 0$ and there exists $\tau_0$ such that $x(t) \ge 0$ for all $t\ge \tau_0$ and for all $x_0\ge 0$. We call a system internally positive if $\tau_0 = 0$.
\end{defn}
The case of $\tau_0 = 0$ is well-studied and it is well-known a system is internally positive system if and only if $C$, $B$, $D$ are nonnegative and $A$ is Metzler (i.e., off-diagonal elements are nonnegative)~\cite{kaczorek2001externally}. Internally positive systems, admit an even easier characterization of the induced norms, since we can exploit positivity of the realization~\cite{rantzer2015ejc, rantzer2016kalman}. The goal of this paper is to develop similar results in the context of (internally) eventually positive systems. A characterization of eventually $\cK$-positive dynamical systems was obtained in~\cite{kasigwa2017eventual}. We reformulate this result using the linear systems language and prove it for completeness.

\begin{prop}\label{prop:ev-pos-dyn}
	Consider the proper cone $\cK$ and the system $\dot x = A x$ with $\lambda_j$ being the eigenvalues of $A$ ordered such that $\Re(\lambda_i)\ge \Re(\lambda_j)$ for all $i\ge j$. \\
	(i) If the system is eventually $\cK$-positive and $A$ is diagonalizable, then $\eta(A)$ is the eigenvalue of $A$ and the corresponding right and left eigenvectors $v$, $w$ can be chosen such that $v\in \cK$, $w\in \cK^\ast$ (and  $\eta(A)= \lambda_1$);\\
	(ii) The system is strongly eventually $\cK$-positive if and only if $\eta(A)$ is a simple eigenvalue and $\eta(A)= \lambda_1 > \Re(\lambda_j)$ for all $j \ge 2$, while the right and left eigenvectors $v$, $w$ of $A$ corresponding to $\lambda_1$ can be chosen such that $v\in \inter(\cK)$, $w\in \inter(\cK^\ast)$.
\end{prop}
\begin{proof} 
	(i) $e^{A t} x$ belongs to $\cK$ for all $t\ge \tau_0$ for any $x \in \cK$. We will employ the point (i) in Theorem~12 from~\cite{kasigwa2017eventual}. Note that our definition of eventually positive systems $\dot x = A x$ is equivalent to $A$ being eventually exponentially $\cK$ nonnegative in the terminology from~\cite{kasigwa2017eventual} (see, Definition~1 in~\cite{kasigwa2017eventual}).
	According to Theorem~12 from~\cite{kasigwa2017eventual} we have that that $\rho(A)$ the spectral radius of $e^{A}$ is its eigenvalue, the matrix $e^A$ has the $\cK$-Perron-Frobenius property, while $e^{A^T}$ has the $\cK^\ast$ has the $\cK^\ast$-Perron-Frobenius property. By definition of these properties (see Definition~3 in~\cite{kasigwa2017eventual}), we have that the left $w$ and right $v$ eigenvectors corresponding to $\rho(e^A)$ can be chosen such that $v\in \cK$, $w\in \cK^\ast$. Furthermore, the eigenvalues $\lambda_i$ of $A$ are such that $e^{\lambda_i}$ are eigenvalues of $e^A$. Therefore, $\rho(e^A) = e^\lambda$ where $\lambda$ is some real eigenvalue of $A$. Since $e^\lambda \ge |e^{\mu}|$ for any other eigenvalue $\mu$ of $A$, we have that $\lambda = \eta(A)$. Finally, since $A$ is diagonalizable $w$, $v$ are the left and right eigenvectors of $A$ corresponding to $\mu$.  
	
	(ii) \emph{Necessity.}  In order to show this point, we will use the proof of Theorem~8 from~\cite{kasigwa2017eventual}, which we reproduce for completeness. For all $t\ge \tau_0$ the flow $e^{A t} x\in\inter(\cK)$  for any nonzero $x\in\cK$, which in the nomenclature of~\cite{kasigwa2017eventual} means that $A$ is eventually exponentially $\cK$-positive. According to Theorem~7 in~\cite{kasigwa2017eventual} this implies that $e^A$ has the strong Perron-Frobenius property (Definition~3 in~\cite{kasigwa2017eventual}), i.e., $\rho(e^A)$ is a simple eigenvalue of $e^A$ such that the corresponding left $w$ and right $v$ eigenvectors can be chosen to lie in $\inter(\cK^\ast)$ and $\inter(\cK)$, respectively. Since the spectra of $A$ and $e^A$ are linked as $\lambda(A) = e^{\lambda(A)}$, this implies that $\rho(e^A) = e^\lambda$ for some real eigenvalue $\lambda$ of $A$. Furthermore, for any other eigenvalue $\mu$ of $A$, we have that $e^\lambda > |e^\mu| = e^{\Re(\mu)}$, therefore $\lambda$ is the spectral abscissa $\eta(A)$ of $A$. Since $\eta(A)$ is a simple and real eigenvalue of $A$ it shares the same eigenvectors with the eigenvalue $\rho(e^A)$  of $e^A$. To summarize $\eta(A)$ is a simple and real eigenvalue of $A$, while its left $w$ and right $v$ eigenvectors can be chosen such that $v\in \inter(\cK)$, $w\in \inter(\cK^\ast)$, which completes the proof.

	\emph{Sufficiency.} Let $v\in\inter(\cK)$, $w\in\inter(\cK^\ast)$ and $\lambda_1$ be real, simple and $\lambda_1 >\Re(\lambda_i)$ for all $i\ge 2$. The matrix $e^{A t}$ can be decomposed as $e^{\lambda_1 t} v w^T + R(t)$, where $\lim_{t \rightarrow \infty}\|e^{-\lambda_1 t} R(t)\|_2  =0$ (cf.~\cite{mezic2015applications}). Then we have
	\[
	 e^{A t} x = e^{\lambda_1 t} \left(v w^T + e^{-\lambda_1 t} R(t)\right)x\,.
	\]
	Since $v\in \inter(\cK)$, $w\in \inter(\cK^\ast)$, $w^T x>0$ for any nonzero $x\in \cK$ and $v w^T x\in \inter(\cK)$ for any $x\in \cK$. Furthermore, since $\Re(\lambda_i) < \lambda_1$ for all $i\ge 2$, there exists a time $\tau_0$ such that $(v w^T + e^{-\lambda_1 t}R(t)) x\in \inter(\cK)$ for all $t\ge\tau_0$. Hence we have that $e^{A t} x \in \inter(\cK)$.
\end{proof}
\begin{rem}
	In the point (i), we can actually require that the spaces of left and right eigenvectors for the dominant eigenvalue $\lambda_1$ of the matrix $A$ are non-degenerate (i.e., the algebraic and geometric multiplicities are equal) instead of $A$ being diagonalizable. This holds if, for instance, $\eta(A)$ is a simple and real eigenvalue of $A$. This can be shown using decompositions in~\cite{mezic2015applications}, however, we consider this as a minor generalization. Furthermore, in what follows we will still at times assume diagonalizability of $A$, which will simplify our analysis. While non-diagonalizable matrices can still be met in applications, at this point we do not focus on this case. 
\end{rem}

\section{Results\label{s:results}}
\subsection{Eventual Positivity and Positivity }
It is known that there exists a cone $\cK$ such that the system $\dot x =  A x$ is strongly $\cK$-positive, if $\lambda_1$ is simple, real and $\lambda_1 > \Re(\lambda_i)$ for all $i \ge 2$, where $\lambda_i$'s are eigenvalues of $A$ (see Remark~3.34 in~\cite{stern1991exponential}). This implies that any strongly eventually $\cK$-positive system is also strongly $\hat \cK$-positive, with respect to a possibly different cone $\hat \cK$. Actually we can state the same for eventual positivity in general.
\begin{thm}
Let the system $\dot x = A x$ be eventually positive with respect to a proper cone $\cK$ and consider  the set  $\widetilde \cK$ defined as follows:
\begin{gather}\label{cone-ev-pos}
\widetilde \cK = \left\{x\in\R^n \Bigl| e^{A t} x \in \cK, \forall t\ge \tau_0 \right\},
\end{gather}
where $\tau_0$ is such that $e^{A t}\cK \subseteq \cK$ for all $t\ge\tau_0$. Then:\\
(i) the set $\widetilde \cK$ is a proper cone; \\
(ii) $e^{A t}\widetilde \cK \subseteq \widetilde \cK$ for all $t\ge 0$.
\end{thm}
\begin{proof}
(i) First we need to verify that $\widetilde \cK$ is a proper cone, if $\cK$ is. For every $x_1$, $x_2 \in \widetilde \cK$ and $\alpha_1$, $\alpha_2 \in \Rnn$ we have that $e^{A t} (\alpha_1 x_1 + \alpha_2 x_2)$ belongs to $\cK$ for all $t\ge \tau_0$, since $e^{A t} \alpha_1 x_1$, $e^{A t}\alpha_2 x_2$ belong to $\cK$ for all $t\ge\tau_0$. Therefore $\alpha_1 x_1 + \alpha_2 x_2$ belongs to $\widetilde \cK$. Similarly we can show that $\widetilde \cK\bigcap - \widetilde \cK = \{0\}$. The set $\widetilde \cK$ also has a non-empty interior since $\cK$ is proper and $\cK\subseteq \widetilde \cK$. It is left to show that $\widetilde \cK$ is closed. Consider a sequence $\{x_n\} \in \widetilde \cK$ for all integer $n$ and $\lim\limits_{n \rightarrow \infty} \|x_n - x\|_2 = 0$. We will show by contradiction that $x\in \widetilde \cK$. Let $x \not\in \widetilde \cK$, then there exists a time $s\ge \tau_0$ such that $e^{ A s} x \not \in \cK$. Since $\cK$ is closed the distance $\min_{z\in \cK}\|e^{A s} x - z \|_2 = \delta$ is larger than zero. Since $\{x_n\}_{n = 1}^{\infty}$ converges to $x$, there exists an integer $N$ such that for all $n\ge N$ we have $\|x_n - x\|_2 \le \delta/(2 \|e^{A s}\|_2)$. Now since $e^{A s}x_n\in \cK$ for all $n$ we have that $\|e^{A s} x_n - e^{A s} x\|_2\ge \delta$ for all $n$. Furthermore, for $n\ge N$ we have that
\begin{gather*}
0<\delta \le \|e^{A s} x_n - e^{A s} x\|_2 \le \| e^{A s}\|_2 \|x_n - x\|_2 \le \delta/2,
\end{gather*}
which is a contradiction, therefore $\widetilde \cK$ is closed.

(ii) We need to show that $e^{A t} \widetilde \cK \subseteq \widetilde \cK$ for all $t\ge 0$. Let $x\in\widetilde \cK$, we need to verify that $y = e^{At }x$ belongs to $\widetilde \cK$ for all $t\ge0$.
According to the definition of $\widetilde \cK$, for all $t\ge 0$, $s\ge \tau_0$ we have that $e^{A (t+s)} x \in \cK$. Consequently, we have that $ e^{A s} (e^{A t}x) = e^{A s} y\in \cK$ and $ y =e^{A t} x\in \widetilde \cK$, which completes the proof.
\end{proof}

This proposition raises a valid question: what is the benefit of studying eventual positivity if we can use the classical positivity, instead? One of the differences between eventual and strict definitions of $\cK$-positivity comes in the freedom of choice of the cones $\cK$. Most of the powerful results for positive systems were derived for cones, which are orthants such as $\Rnn^n$. For instance, if the system is orthant-positive then stability analysis and computation of Lyapunov function is significantly simplified. We will show in the next section that strongly eventually $\Rnn^n$-positive systems retain some properties of $\Rnn^n$-positive systems. At the same time we can also use $\widetilde \cK$-positivity to the full extent. While it is an open question if there exists a state-space transformation such that $\cK$-positive system becomes strongly $\Rnn^n$-positive, we can easily find a transformation transforming a system into strongly eventually $\Rnn^n$-positive under some mild restrictions.
\begin{prop} \label{prop:dom-eig}
	Consider the system $\dot x = A x$ with $\lambda_j$ being the eigenvalues of $A$. Let $\lambda_1$ be simple, real, negative and $\lambda_1 > \Re(\lambda_j)$ for all $j \ge 2$. Then there exists an invertible matrix $S$ such that the system $\dot z = S^{-1} A S z$ is strongly eventually $\Rnn^n$-positive.  
\end{prop}
\begin{proof}
	Let $v$ and $w$ be the right and left eigenvectors corresponding to the dominant eigenvalue $\lambda_1$ of $A$. According to Proposition~\ref{prop:ev-pos-dyn}, we need to show that there exists an invertible matrix $S$ such that $S^{-1} v$ and $S^T w $ are positive. Without loss of generality, we assume that the first entry of $w$ is nonzero. We can find a transformation $S$ such that $w^T S  = \bfone^T/n$ and $S \bfone  = v$ as follows
	\begin{gather*}
	S =  I_n/n + \begin{pmatrix}
	v -  \bfone/n & 0_{n\times n-1}
	\end{pmatrix}+\begin{pmatrix}
	\frac{(\bfone - w)^T}{w(1) n} \\
	0_{n-1\times n}
	\end{pmatrix} + S_0,
	\end{gather*}
	where $I_n$ is the $n\times n$ identity matrix, $0_{k\times m}$ is the $k\times m$ zero matrix, and $S_0$ is a zero matrix except for one entry, where $S_0(1,1) = (-1/w(1) +  w^T \bfone/(w(1) n))$. We verify the claim by direct calculations: 
	\begin{multline*}
	S \bfone =  \bfone/n + v - \bfone/n + \frac{(\bfone - w)^T}{w(1) n} \bfone \\+ (-1/w(1) +  w^T \bfone/(w(1) n)) = v
	\end{multline*}
	Similarly 
	\begin{multline*}
	w^T S =  w^T/n  + \begin{pmatrix}
	w^T v -  \frac{w^T\bfone}{n} & 0_{1\times n-1}
	\end{pmatrix} + \\ \frac{(\bfone - w)^T}{n} + 
	\begin{pmatrix}
	-1 + \frac{w^T \bfone}{n} & 0_{1\times n-1}
	\end{pmatrix}  =  \bfone^T /n
	\end{multline*}
	In this case new dominant eigenvectors are $\widetilde v = \bfone $ and $\widetilde w = \bfone/n$.  Since the eigenvalues do not change under the similarity transformation, the dominant eigenvalue of $S^{-1} A S$ is simple and real. Now invoking point (ii) in Proposition~\ref{prop:ev-pos-dyn} proves the result.
\end{proof}

Another piece of the puzzle that simplifies the analysis of $\Rnn^n$-positive systems is the invariance of these systems with respect to $\Rnn^n$. While eventually $\Rnn^n$-positive systems are invariant with respect to the cone $\widetilde \cK = \left\{x\in\R^n \Bigl| e^{A t} x \in \Rnn^n, \forall t\ge \tau_0 \right\}$, studying the cone $\widetilde \cK$ is not an easy task. However, under additional assumptions on $A$ we can derive a whole family of invariant cones. Suppose that the matrix $A$ is diagonalizable, $w^i$, $v^i$ are the left and right eigenvectors of $A$ corresponding to $\lambda_i$, where $\lambda_1 > \Re(\lambda_i)$ for all $i\ge 2$. We introduce the following family of cones:
\begin{multline}\label{ice-cream-cones}
\cK_{\alpha}=  \Bigg\{ y \in \R^n \Bigl| \left(\sum\limits_{i =2}^n \alpha_i |(w^i)^T y|^2\right)^{1/2} \le (w^1)^T y \Bigg\}\,, 
\end{multline}
where $\alpha_i$ are positive scalars and are chosen a priori. Every set $\cK_\alpha$ is a Lorentz cone subject to a transformation $T = \begin{pmatrix} w^1 & \sqrt{\alpha_2} w^2 & \cdots & \sqrt{\alpha_n}w^n\end{pmatrix}^T$, which is invertible since $w^i$ are linearly independent. We have the following result:
\begin{thm}\label{prop:pos-cone}
	Consider the system $\dot x = A x$ with a diagonalizable $A$ with eigenvalues $\lambda_j$. Let $\lambda_1$ be simple, real and negative, and $\lambda_1 > \Re(\lambda_j)$ for all $j\ge 2$. Then the system is $\cK_\alpha$-positive for any positive scalars $\alpha_i$ with $i=2,\dots,n$.
\end{thm}
\begin{proof} 
	Let $y = e^{A t}x$ for $t>0$, then 
	\begin{align*}
	& ((w^1)^T y)^2 - \sum\limits_{i =2}^n \alpha_i |(w^i)^T y|^2 \\
	& \qquad  = ((w^1)^T e^{A t} x)^2   - \sum\limits_{i =2}^n\alpha_i |(w^i)^T e^{A t} x|^2\\
	& \qquad = e^{2 \lambda_1 t} ((w^1)^T x)^2 - \sum\limits_{i =2}^n\alpha_i |e^{\lambda_i t}|^2  |(w^i)^T x|^2 \\
	& \qquad = e^{2 \lambda_1 t} \Bigl( ((w^1)^T x)^2 - 
	\sum\limits_{i =2}^n\alpha_i |e^{(\lambda_i-\lambda_1) t}|^2 |(w^i)^T x|^2\Bigl).
	\end{align*}
	Since $\lambda_1 >\Re(\lambda_i)$ for all $i>1$ we have that $|e^{(\lambda_i-\lambda_1) t}|^2 < 1$ for all $t>0$, which in turn implies that
	\begin{multline*}
	((w^1)^T y)^2 - \sum\limits_{i =2}^n \alpha_i |(w^i)^T y|^2 \ge \\e^{2 \lambda_1 t} \left( ((w^1)^T x)^2 - \sum\limits_{i =2}^n \alpha_i |(w^i)^T x|^2\right), 
	\end{multline*}
	and $y = e^{At} x$ belongs to $\cK_\alpha$ if $x$ does.
\end{proof}

\subsection{Lyapunov Functions}\label{s:stability} 

In the context of $\Rnn^n$-positive systems it is common to define Lyapunov functions on the invariant sets of the system such as $\Rnn^n$, for example. In the case of eventually $\Rnn^n$-positive systems, we need to use the cones $\cK_{\alpha}$ in order to introduce Lyapunov functions.

\begin{thm} \label{prop:lyap-ev-pos}
Let $\dot x = A x$ be a strongly eventually positive system, let $\lambda_1$ be a simple eigenvalue of the diagonalizable matrix $A$ such that $\lambda_1 >\Re(\lambda_j)$ for all $j\ge 2$ and let 
$v^1$, $w^1$ be the right and left eigenvectors of $A$ corresponding to $\lambda_1$. Then 

(i) the system $\dot x = A x$ is asymptotically stable if and only if there exists a positive vector $\xi$ such that $\xi^T A  \ll 0$;

(ii) for any positive vector $\xi$ such that $\xi^T A \ll 0$ and any $\alpha_i>0$ such that $\cK_\alpha\subseteq\Rnn^n$ the function $V_s(x) = \xi^T x$ is a Lyapunov function on $\cK_\alpha$. Moreover, the function $V_s(x) = x^T w^1$ 
is a Lyapunov function on $\cK_\alpha$ for any $\alpha_i>0$ including values for which $\Rnn^n \subseteq \cK_\alpha$;

(iii) there exist $\alpha_i>0$ such that $\cK_\alpha\subseteq\Rnn^n$ and $V_m(x) = \max\{x_1/v^1_1, \dots, x_n/v^1_n\}$ is a Lyapunov function on $\cK_\alpha$. 
\end{thm}
\begin{rem}
	If $\tau_0=0$, i.e. the system $\dot x = A x$ is positive then we can replace $\cK_{\alpha}$ with $\Rnn^n$ and the function $V_m$ always exists.
\end{rem}
\begin{proof}
(i) This result is taken from~\cite{olesky2009m}, while the proof is presented for completeness.

\emph{Necessity.}  If the system is asymptotically stable, then we can pick $\xi = w^1$, and $(w^1)^T A = \lambda_1 (w^1)^T \ll 0$. 

\emph{Sufficiency.} Let $\xi\gg 0$ be such that $\xi^T A \ll 0$, then $\lambda_1 \xi^T v^1 = \xi^T A v^1$. Since $\xi^T v^1$, $v^1$ are positive and $\xi^T A$ is negative, the eigenvalue $\lambda_1$ has to be negative. Since $\lambda_1$ is the dominant eigenvalue, the matrix $A$ is asymptotically stable.

(ii) We have that $V_s(x) > 0$ for all $x$ on $\{x | \xi^T x > 0 \}$. Furthermore, $\dot V_s(x) = \xi^T A x  < 0 $ for all $x\in \{x | \xi^T A x < 0\}$. It is straightforward to verify that the set $\cD =\{x | \xi^T A x < 0\}\cap \{x | \xi^T x > 0 \}$ contains $\Rnn^n\backslash\{0\}$ for all $\xi \gg 0$ such that $\xi^T A \ll0$. Since $V_s(0) =0$, $V_s(x)$ is a Lyapunov function on $\cK_\alpha$ for any $\alpha_i>0$ such that $\cK_\alpha \subseteq \Rnn^n \subseteq \cD$. The second part of the statement is straightforward. 

(iii) Note that if the maximum of $V_m(x)$ at time $t$ is achieved at the index $i$ such that $x_i/v_i^1\ge x_j/v_j^1$ for all $j$, then $\dot V_m(x) = \sum\limits_{j = 1}^n A_{i j} x_j/v^1_j$. Consider the set
\begin{gather*}
\cD_0 = \left\{x \in\Rnn^n \Bigl| \sum\limits_{j = 1}^n A_{i j} x_j < 0~\forall i\right\}.
\end{gather*}
It is clear that the ray $\{ x \in R^n |x = \beta v^1, \beta\in\Rnn \}$ lies in $\cD_0$. Therefore there exists $\cK_\alpha$ with large enough $\alpha_i >0$ such that for all $x\in\cK_\alpha$, $x$ also belongs to $\cD_0$. Therefore, $V_m(x)$ is negative on $\cK_{\alpha}$ and it is a Lyapunov function on $\cK_{\alpha}$.
\end{proof}

\subsection{Eventually Positive Input-Output Systems}
\label{s:ev-pos-io}
In this section, we consider control systems~\eqref{sys:in-out} and study \emph{internally eventually positive systems} (with respect to $\Rnn^n$), which eventually behave like internally positive systems. The complete characterization of such systems is provided by the following result. 
 
\begin{thm} \label{prop:ev-non-cont}
Consider the system~\eqref{sys:in-out} with $B\in\R^{n\times m}$, $C\in\R^{k\times n}$, and let $D = 0$. The system is internally eventually positive if and only if the following conditions hold:

(i) the system $\dot z = A z$ is eventually positive;

(ii) $e^{At}B\ge 0$ for all $t>0$; 

(iii) $C e^{At}\ge 0$ for all $t>0$.
\end{thm}
\begin{proof} \emph{Necessity}. Let the system be internally eventually positive. Then there exists $\tau_0$ such that $e^{At} x_0$ for all $t\ge \tau_0$ and $x_0\ge 0$. This implies that $e^{A t} \ge 0$ for all $t\ge \tau_0$ and consequently the system $\dot z = A z$ is eventually positive. The flow $C e^{A t} x_0$ has to be nonnegative for all $t\ge 0$, hence $C e^{A t}$ is nonnegative for all $t\ge 0$ and (iii) is shown. It is left to we verify the condition (ii). Let $x_0 = 0$, then we have 
\begin{equation*}
x(t)=\int_0^{t} e^{A(t-\tau)} B u(\tau) d\tau \ge 0
\end{equation*}
and since this holds for all $u\ge 0$, we have $e^{A t} B \ge 0$ for all $t\ge 0$. 

\emph{Sufficiency}. Decompose the flow of the system into:
\begin{gather*}
y(t) =  C x(t) = C e^{A t} x_0 + C \underbrace{\int_0^t  e^{A (t-\tau)} B u(\tau) d \tau}_{I(t)}.
\end{gather*}

The integral $I(t)$ is nonnegative for all $t$, since $e^{A (t-\tau)} B$, $u(\tau)$ are nonnegative for all $t\ge \tau\ge 0$ due to condition (ii). The vector $e^{A t}x_0$ is nonnegative for all $t\ge \tau_0$, where $\tau_0$ is the exponential index of $\dot z = A z$, hence $x(t)\ge 0$ for all $t\ge \tau_0$.

The matrix $C$ is nonnegative due to (iii), hence $C I(t)$ is also nonnegative. Finally, since $C e^{At}$ is nonnegative for $t \ge 0$, the function $y(t)$ is nonnegative for all $t\ge 0$.
\end{proof}

As we mentioned before, we cannot easily check if the system is eventually positive, since we have only necessary conditions. Moreover, checking conditions (ii) and (iii) may be computationally hard. Alternatively, we can use the following corollary.

\begin{cor}\label{cor:ev-pos}
Consider the system~\eqref{sys:in-out}, let $D$ be nonnegative, let the system $\dot z = A z$ be strongly eventually positive, i.e., there exists $\tau_0\ge 0$ such that $e^{A t}\Rnn^n\subseteq\Rnn^n$ for all $t\ge \tau_0$ being its exponential index. Let one of the following conditions be fulfilled

(i) Every column of $B$ (respectively, $C^T$) lie in the set $e^{A\tau_0} \Rnn^n$ (respectively, $e^{A^T\tau_0} \Rnn^n$); 

(ii) Every column of $B$ belongs to $\cK_\alpha$, while every column of $C^T$ belongs to $\cK_\alpha$ for some positive vector $\alpha$.

Then the system~\eqref{sys:in-out} is internally eventually positive. 
\end{cor}

We had to sacrifice some freedom in the choice of $B$, and $C$ matrices, however, we were able to obtain more freedom in the choice of the drift matrix $A$. Hence for every fixed $\tau_0$, our class of internally eventually positive systems is not a superset or a subset of the class of internally positive systems. At the same time the union over all $\tau_0\ge0$ provides a much larger class than the class of positive systems. 
\subsection{Energy Functions} 
In the context of model reduction we study the so-called energy functions
\[
\cQ_p = \|y\|_{\Lp[0,+\infty)}, \quad
\cC_p = \inf\limits_{\begin{smallmatrix}
 u\in\Lp(-\infty,0], \\
 u(\cdot) \ge 0 \\
 x(0)=x_0,\\
 x(-\infty) = 0
 \end{smallmatrix}} \|u\|_{\Lp(-\infty,0]}.
\]

In the context of internally positive systems the energy functions for $p=1,\infty$ were studied in~\cite{Sootla2012positive}, while the case $p=2$ was studied in~\cite{grussler2012symmetry}. One of the properties used in both papers was that $A$ has a nonpositive inverse and $e^{A t}$ is nonnegative for all $t>0$, which does not hold for internally eventually positive systems. But we still have the following result.

\begin{prop}\label{lem:ss-pos} Let the system~\eqref{sys:in-out} be internally eventually positive with an asymptotically stable $A$. Then the matrices $-A^{-1} B$ and $-C A^{-1}$ are nonnegative.
\end{prop}
\begin{proof} It is straightforward to verify that 
\begin{gather*}
\int_{0}^{\infty} e^{A t} B dt = A^{-1} e^{At}\Bigl|_{0}^{\infty} B = - A^{-1} B
\end{gather*}

Since $e^{A t}B$ is nonnegative for all $t$, the integral, which is equal to $-A^{-1} B$, is nonnegative as well. The fact that $-C A^{-1}$ is nonnegative is shown in a similar manner.
\end{proof}

Note that the vector $- A^{-1} B u_0$ for any $u_0$ can be seen as a steady state response toward the control signal $u_0$. Hence naturally, if $-A^{-1} B$ is nonnegative, then the steady state response to a constant nonnegative input is nonnegative. Therefore, we have an externally positive system such that its state $x(t)$ is asymptotically positive. The matrix $-C A^{-1}$ has of course a control theoretic dual interpretation to $- A^{-1} B$. Another interpretation of these matrices is through the defined above energy functions. Let $B\in\R^{n\times 1}$ and $C\in\R^{1\times n}$ and compute the observability energy function $\cO_1$ as follows:
\[
\cO_1 = \|y\|_{\Lone[0,+\infty)} = \int\limits_0^{+\infty} |C e^{A t} x_0| d t = - C A^{-1} x_0,
\]

while for $\cC_\infty$ we only have a lower bound. Let $p$ be the vector such that $- p^T A^{-1} B = 1$. Then using straightforward computations we have
\begin{multline*}
p^T x_0 = p^T \int\limits_{-\infty}^0 e^{A t} B u(t) dt \le \\
p^T \int\limits_{-\infty}^0 e^{A t} B  dt \cdot \|u \|_{\Linf(-\infty,0]}  = \|u \|_{\Linf(-\infty,0]},
\end{multline*}
and hence $\cC_\infty \ge p^T x_0$. The equality cannot be achieved for all $x_0$, since some states are not reachable with positive control signals and hence $\cC_\infty = +\infty$.

We can also devise some properties of the classical energy functions $\cO_2$, $\cC_2$, which are computed as 
\begin{subequations}\label{eq:lyap}
\begin{eqnarray}
 \cO_2(x_0) = \langle x_0, Q x_0\rangle^{1/2}      &  A^T Q  + Q A + C^T C = 0, \label{eq:lyap-cont}\\
 \cC_2(x_0) = \langle x_0, P^{-1} x_0\rangle^{1/2} & P A^T + A P + B B^T = 0.   \label{eq:lyap-obsv} 
\end{eqnarray}   
\end{subequations}

The solutions to~\eqref{eq:lyap-cont},~\eqref{eq:lyap-obsv} have the following integral forms:
\[
P = \int\limits_0^{\infty} e^{A t} B B^T e^{A^T t} dt,\quad Q = \int\limits_0^{\infty} e^{A^T t} C^T C e^{A t} dt.
\]
Since the system is eventually internally positive, the matrix $e^{A t} B$ is nonnegative for all $t\ge 0$, therefore $P$ is a nonnegative matrix. 
Moreover, if $\dot x = Ax$ is strongly eventually positive there exists a time $\tau_0$ such that $e^{A t}$ is a positive matrix for all $t\ge \tau_0$. If the matrix $B B^T$ is  nonnegative and nonzero, then the matrix $e^{A t} B B^T e^{A^T t}$ is irreducible and nonnegative for all $t\ge \tau_0$. This implies that $P$ is nonnegative and irreducible under the premise of Corollary~\ref{cor:ev-pos}. Similarly we can show that solution $Q$ to~\eqref{eq:lyap-obsv} is also nonnegative and irreducible. 

\subsection{Induced Norms}
It is clear that Proposition~\ref{prop:induced-norms} still holds and the norms can be computed using linear algebra, since we are dealing with an externally positive system. Moreover, using Proposition~\ref{lem:ss-pos}, we can extend internally positive systems results in~\cite{briat2013robust} to internally eventually positive systems.
\begin{prop}
Let a realization of the system $G$ satisfy the premise of Corollary~\ref{cor:ev-pos} and hence be internally eventually positive. Then 

(i) the matrix $A$ is Hurwitz and $\|G\|_{\rm \infty-ind} <\gamma$ if and only if there exists $\zeta > 0$ such that 
\begin{gather} \label{ineq:linf-norm}
\begin{pmatrix}
A & B \\ 
C & D
\end{pmatrix}\begin{pmatrix}
\zeta \\
\bfone
\end{pmatrix} \ll \begin{pmatrix}0 \\ \gamma \bfone \end{pmatrix} .
\end{gather}

(ii) the matrix $A$ is Hurwitz and $\|G\|_{\rm 1-ind} <\gamma$ if and only if there exists $\xi > 0$ such that 
\begin{gather} \label{ineq:l1-norm}
\begin{pmatrix}
A & B \\ 
C & D
\end{pmatrix}^T \begin{pmatrix}
\xi \\
\bfone
\end{pmatrix} \ll \begin{pmatrix}0 \\ \gamma \bfone \end{pmatrix}.
\end{gather}
\end{prop}

\begin{proof} 
(i) Let the matrix $A$ be Hurwitz  and $\|G\|_{\rm \infty-ind} <\gamma$. Due to Proposition~\ref{prop:induced-norms} we have that $\|G(0)\|_{\rm \infty-ind} < \gamma$ or $\|D - C A^{-1} B\|_{\rm \infty-ind} < \gamma$. Since $-A^{-1}B$ is nonnegative due to Proposition~\ref{lem:ss-pos}, if the matrix $A$ is Hurwitz, the condition $\|G\|_{\rm \infty-ind} <\gamma$ is fulfilled if and only if
\begin{gather}\label{eq:linf-norm}
(D - C A^{-1} B)\bfone \ll \gamma \bfone.
\end{gather}

\emph{Necessity.} Since $A$ is Hurwitz, there exists $x \gg 0$ such that $A x \ll 0$ and let $\zeta  = x - A^{-1} B \bfone$. According to Proposition~\ref{lem:ss-pos}, $-A^{-1} B$ is nonnegative, therefore we have $\zeta \ge x \gg 0$, which implies that $A \zeta + B \bfone= A x \ll 0$. Due to~\eqref{eq:linf-norm}, we have that $(D + C (-A^{-1} B))\bfone \ll \gamma \bfone$, therefore for $x$ with a sufficiently small norm we get $(D + C (x-A^{-1} B))\bfone \ll \gamma \bfone$ and hence $C \zeta + D \bfone \ll \gamma\bfone$. Therefore, the inequality follows~\eqref{ineq:linf-norm}.

\emph{Sufficiency.} Let the inequality~\eqref{ineq:linf-norm} hold. The inequality $A \zeta + B\bfone \ll 0$ implies that $A$ is Hurwitz. Multiplying $A \zeta + B\bfone \ll 0$ with a nonnegative matrix $-C A^{-1}$ (according to Proposition~\ref{lem:ss-pos}) from the left gives $-C\zeta - CA^{-1} B \bfone \le 0$. Adding this to the inequality $C \zeta + D \bfone \ll \gamma \bfone$ gives~\eqref{eq:linf-norm} and consequently $\|G\|_{\rm \infty-ind} <\gamma$.

The point (ii) is shown in a similar manner.
\end{proof}
\section{Illustrative Example\label{s:example}} 
Consider the matrix
\begin{gather*}
A = \begin{pmatrix}
-6  &  10&  4\\
-7  &  2 & 12 \\
3  & -3 & -4
\end{pmatrix},
\end{gather*}
whose eigenvalues are $\lambda_1 = -0.2924$, $\lambda_{2,3} = -3.8538 \pm \imath 8.9941$. The eigenvectors corresponding to the dominant eigenvalue $\lambda_1$ are 
\begin{align*}
v_1 &= \begin{pmatrix}
0.8460 & 0.3091 & 0.4345
\end{pmatrix}^T, \\
w_1 &= \begin{pmatrix}
0.3474 & 0.3167 & 1.400
\end{pmatrix}^T.
\end{align*}
 This implies that the system $\dot x = A x$ is strongly eventually positive. We will analyse this matrix for various applications and different settings.

\emph{Eventually Positive System as a System With Slow Positive Dynamics:}
Consider the linear system
\begin{gather*}
\begin{pmatrix}
\dot x_1 \\  \dot x_2 \\ \varepsilon \dot x_3
\end{pmatrix} = A \begin{pmatrix}
x_1 \\ x_2 \\ x_3
\end{pmatrix} 
\end{gather*}
with $0<\epsilon \ll 1$. We can eliminate the variable $x_3$ using singular perturbation theory and obtain the system
\begin{align*}
&\begin{pmatrix}
\dot {\tilde x}_1 \\ \dot {\tilde x}_2
\end{pmatrix}  = \tilde A \begin{pmatrix}
\tilde x_1 \\ \tilde x_2
\end{pmatrix}, \text{ where} \\
&\tilde A  = \begin{pmatrix}
-6  &  10 \\
-7  &  2 \end{pmatrix} 
+ \frac{1}{4} \begin{pmatrix}  4 \\ 12
\end{pmatrix} 
\begin{pmatrix}
3 \\ -3
\end{pmatrix}^T = \begin{pmatrix}
-3  &  7\\
2  & -7
\end{pmatrix}.
\end{align*}
The drift matrix $\tilde A$ of the reduced order system is Metzler, while the matrix $A$ is not. It can be verified that the flow $x_1(t)$, $x_2(t)$ converges to $\tilde x_1(t)$, $\tilde x_2(t)$, while $x_3(t)$ converges to a constant. Moreover, after some time $\tau_0>0$, the flow of the full system $x(t)$ becomes positive. At the same time, since the matrix $A$ has positive dominant eigenvectors corresponding to $\lambda_1$, we can reach the same conclusion with $\varepsilon = 1$. This example illustrates that eventual positive can originate from systems with slow positive dynamics and the transient is caused by fast dynamics. However, eventual positivity is not limited to this case.

\emph{Lyapunov Functions:} Proposition~\ref{prop:lyap-ev-pos} shows that, for eventually positive systems, there exist sum-separable functions $V_s(x)$ on invariant sets containing $\Rnn^n$. This is not the case for the so-called max-separable Lyapunov functions $V_m = \max\{x_1/\zeta_1,\dots x_n/\zeta_n\}$, where $\zeta\gg 0$ and $A\zeta\ll 0$. Similarly, sum-separable Lyapunov functions of the form $V_d(x) = \sum\limits_{i =1} p_i x_i^2$ on $\R^n$ may not exist. This phenomenon can be illustrated on the matrix $A$. Note that $V_d$ is a Lyapunov function if and only the matrix $A^T P + P A$ is negative definite with $P=\diag{p_1, \dots, p_n}$. Since the entry $\{2,2\}$ is positive, there exists no diagonal matrix $P$ such that $A P + P A^T$ is negative definite. The argument for the max-separable function $V_m(x)$ is similar but more subtle. Let at time $t$ the maximum be attained at the index $i$ and let the function be differentiable at $x$. Then $\dot V_m(x) = \sum \limits_{j = 1}^n A_{i j} x_j/\zeta_i$ should be negative for all $x\in\Rnn^n$ such that $x_i/\zeta_i > x_j/\zeta_j$, which includes $e^i$, hence $A_{i i} < 0$. 

\begin{figure}[t]\centering
	\includegraphics[width = .48\columnwidth]{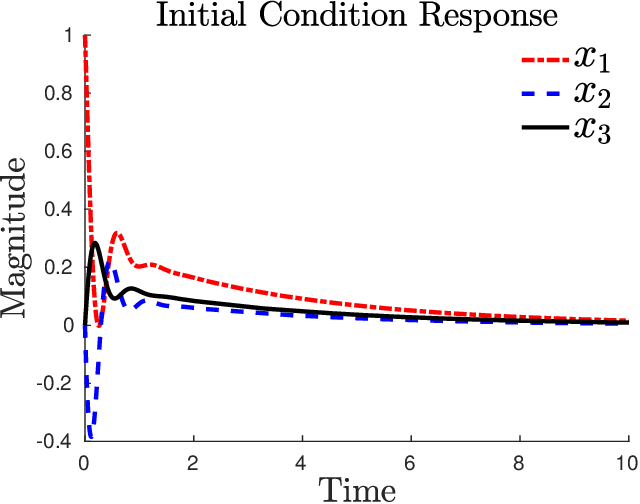}
	\includegraphics[width = .48\columnwidth]{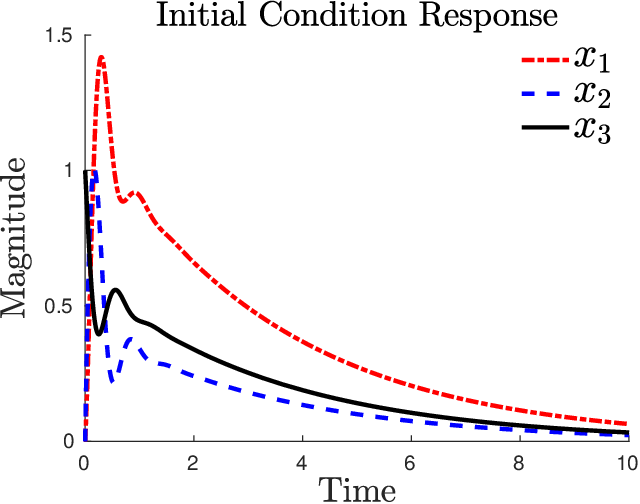}
	\caption{Trajectories of the dynamical system  $\dot x = A x$ with initial conditions $x_{0 1} = (1~~0~~0)^T$ (left panel) and $x_{0 2} = (0~~0~~1)^T$ (right panel).} \label{fig:trajectories}
\end{figure}

\emph{Trajectories of the Dynamical System:} We compute two trajectories of the system $\dot x = Ax$ with initial conditions $x_{0 1} = \begin{pmatrix}
1 & 0 & 0
\end{pmatrix}^T$, $x_{0 2} = \begin{pmatrix}
0 & 0 & 1
\end{pmatrix}^T$, which are depicted in Figure~\ref{fig:trajectories}. The trajectory of $x_2$ is always positive with the initial condition $x_{0 1}$, while the trajectories of $x_1$ and $x_3$  are becoming negative for some time, however, the trajectory eventually enters the orthant $\Rp^3$ and stays for all remaining times $t$. For the initial condition $x_{0 2}$, the picture is different since all trajectories are positive for all $t\ge0$. This means that $x_{02}\in e^{-A\tau_0} \cK$. 

\emph{Impulse Responses of the Control System:} Consider now the system
\begin{align*}
\dot x& = A x + B u, \\
y &= C x,
\end{align*}
where matrix $A$ is defined above and the system is single-input-single-output. The main question in this paragraph is how to choose the matrices $B$ and $C$ so that the system becomes externally positive, i.e., $C e^{At} B$ is positive for all $t\ge0$. We can answer this question using the simulations in the previous paragraph. Indeed, if we choose $B = \begin{pmatrix}
b & 0 & 0
\end{pmatrix}^T$ for some positive scalar $b$, then the output measuring the state $x_2$ will be positive for all $t\ge0$ given a nonnegative input $u(t)$. If we choose $B = \begin{pmatrix}
0 & 0 & b
\end{pmatrix}^T$ for some positive scalar $b$, then the output from all the states will be nonnegative for all $t\ge0$ given a nonnegative input $u(t)$.

\section{Conclusion\label{s:con}}
In this paper, we considered eventually positive linear systems, and their extensions to the input-output setting. While some of the properties of positive systems are retained, analysis of and certificates for eventually positive systems are much harder than for positive systems. Nevertheless, the theoretical development of this class of systems can offer simpler analysis and control tools for linear systems.

The idea of eventual positivity is also extended to nonlinear dynamical systems~\cite{sootla2015koopman} and semigroups of linear operators~\cite{daners2016eventually}, where the certificates for eventual positivity are similar and derived using eigenfunctions/modes of the Koopman operator and eigenvectors of the semigroup. While the input-output extensions in both cases will be of most interest for control community, these extensions are technically more involved than their autonomous counterparts. 


\end{document}